\documentclass[a4paper,12pt]{article}
 
\usepackage{latexsym,amssymb,amsmath, amsthm}  
 
\newtheorem{thm}{Theorem}
\newtheorem{lem}{Lemma}

\newtheorem{cor}{Corollary}
 
\newtheorem{rem}{Remark}

\begin{document}

\title{Generating functions of orthogonal polynomials in higher dimensions}

\author{Hendrik De Bie\\\small{e-mail: hendrik.debie@ugent.be} \and Dixan Pe\~na Pe\~na\\\small{e-mail: dpp@cage.ugent.be} \and Frank Sommen\\\small{e-mail: fs@cage.ugent.be}}

\date{\normalsize{Clifford Research Group, Department of Mathematical Analysis\\Faculty of Engineering and Architecture\\Ghent University\\Galglaan 2, 9000 Gent, Belgium}}

\maketitle

\begin{abstract}
\noindent In this paper two important classes of orthogonal polynomials in higher dimensions using the framework of Clifford analysis are considered, namely the Clifford-Hermite and the Clifford-Gegenbauer polynomials. For both classes an explicit generating function is obtained.\vspace{0.2cm}\\
\noindent\textit{Keywords}: Orthogonal polynomials; Cauchy-Kowalevski extension theorem; Fueter's theorem; Clifford-Hermite polynomials; Clifford-Gegen- bauer polynomials\vspace{0.1cm}\\
\textit{Mathematics Subject Classification}: 30G35, 33C45, 33C50.
\end{abstract}

\section{Introduction}

Clifford analysis is a refinement of harmonic analysis: it is concerned among others, with the study of functions in the kernel of the Dirac operator, a first order operator that squares to the Laplace operator. In its study, special sets of orthogonal polynomials play an important role, see e.g. \cite{DBR} for applications to generalized Fourier transforms and \cite{BDSS} for wavelet transforms. In particular the so-called Clifford-Hermite and Clifford-Gegenbauer polynomials \cite{S2} have received considerable attention. Our main aim in this paper is to find explicit generating functions for both sets of polynomials.

Let us start by introducing the necessary notations and definitions. We denote by $\mathbb{R}_{0,m}$ the real Clifford algebra generated by the standard basis $\{e_1,\ldots,e_m\}$ of the Euclidean space $\mathbb R^m$ (see \cite{Cl}). The multiplication in $\mathbb{R}_{0,m}$ is determined by the relations 
\begin{equation*}
e_je_k+e_ke_j=-2\delta_{jk},\quad j,k=1,\dots,m
\end{equation*}
and a general element $a\in\mathbb R_{0,m}$ may be written as
\[a=\sum_{A\subset M}a_Ae_A,\quad a_A\in\mathbb R,\]
where $e_A=e_{j_1}\dots e_{j_k}$ for $A=\{j_1,\dots,j_k\}\subset M=\{1,\dots,m\}$, with $j_1<\dots<j_k$. For the empty set one puts $e_{\emptyset}=1$, which is the identity element. 

An important part of Clifford analysis is the study of so-called monogenic functions (see e.g. \cite{BDS,DSS,GuSp}). They are defined as follows. A function $f:\Omega\rightarrow\mathbb{R}_{0,m}$ defined and continuously differentiable in an open set $\Omega$ in $\mathbb R^{m+1}$ (resp. $\mathbb R^m$), is said to be monogenic if
\[(\partial_{x_0}+\partial_{\underline x})f=0\quad(\text{resp.}\;\partial_{\underline x}f=0)\;\;\text{in}\;\;\Omega,\] 
where $\partial_{\underline x}=\sum_{j=1}^me_j\partial_{x_j}$ is the Dirac operator in $\mathbb R^m$. Note that the differential operator $\partial_{x_0}+\partial_{\underline x}$, called generalized Cauchy-Riemann operator, provides a factorization of the Laplacian, i.e.
\[\Delta=\sum_{j=0}^m\partial_{x_j}^2=(\partial_{x_0}+\partial_{\underline x})(\partial_{x_0}-\partial_{\underline x}).\]
Thus monogenic functions, like their lower dimensional counterpart (i.e. holomorphic functions), are harmonic.

In this paper we shall deal with two fundamental techniques to generate special monogenic functions: the Cauchy-Kowalevski extension theorem (see \cite{BDS,DSS,S}) and Fueter's theorem (see \cite{F,Q,Sce,S3}). 

The CK-extension theorem states that: \textit{Every $\mathbb R_{0,m}$-valued function $g(\underline x)$ analytic in the open set $\,\underline\Omega\subset\mathbb R^m$ has a unique monogenic extension given by
\begin{equation}\label{CKf}
\mathsf{CK}[g(\underline x)](x_0,\underline x)=\sum_{n=0}^\infty\frac{(-x_0)^n}{n!}\,\partial_{\underline x}^ng(\underline x),
\end{equation}
and defined in a normal open neighborhood $\Omega\subset\mathbb R^{m+1}$ of $\,\underline\Omega$.} 

Fueter's theorem discloses a remarkable connection existing between holomorphic functions and monogenic functions. It was first discovered by R. Fueter in the setting of quaternionic analysis (see \cite{F}) and later generalized to higher dimensions in \cite{Q,Sce,S3}. For further works on this topic we refer the reader to e.g. \cite{CoSaF,CoSaF2,KQS,D,DS1,DS2,DS3,QS}. 

Throughout the paper we assume $P_k(\underline x)$ to be a given arbitrary homogeneous monogenic polynomial of degree $k$ in $\mathbb R^m$. In this paper we are concerned with the following generalization of Fueter's theorem obtained in \cite{S3}.

\textit{Let $h(z)=u(x,y)+iv(x,y)$ be a holomorphic function in some open subset $\Xi$ of the upper half of the complex plane $\mathbb C$. Put $\underline\omega=\underline x/r$, with $r=\vert\underline x\vert$, $\underline x\in\mathbb R^m$. If $m$ is odd, then the function
\begin{equation*}
\mathsf{Ft}\left[h(z),P_k(\underline x)\right](x_0,\underline x)=\Delta^{k+\frac{m-1}{2}}\bigl[\bigl(u(x_0,r)+\underline\omega\,v(x_0,r)\bigr)P_k(\underline x)\bigr]
\end{equation*}
is monogenic in $\Omega=\{(x_0,\underline x)\in\mathbb R^{m+1}:\;(x_0,r)\in\Xi\}$.} 

Define now the following differential operators
\begin{align*}
D_+&=2\underline x-\partial_{\underline x}\\
 D_\alpha&=2(\alpha+1)\underline x-\left(1-\vert\underline x\vert^2\right)\partial_{\underline x},\quad\alpha\in\mathbb R.
 \end{align*}
We shall deal with the so-called Clifford-Hermite and Clifford-Gegenbauer polynomials, which are defined as
\begin{align*}
H_{n,m}(P_k)(\underline x)&=D_+^nP_k(\underline x),\\
C_{n,m}^{(\alpha)}(P_k)(\underline x)&=D_\alpha D_{\alpha+1}\cdots D_{\alpha+n-1}P_k(\underline x).
\end{align*}
It is easy to verify that 
\begin{align*}
H_{n,m}(P_k)(\underline x)&=H_{n,m,k}(\underline x)P_k(\underline x),\\
C_{n,m}^{(\alpha)}(P_k)(\underline x)&=C_{n,m,k}^{(\alpha)}(\underline x)P_k(\underline x),
\end{align*}
where $H_{n,m,k}(\underline x)$ and $C_{n,m,k}^{(\alpha)}(\underline x)$ are polynomials  with real coefficients in the vector variable $\underline x$ of degree $n$. Moreover, we have from \cite{DSS} (see also \cite{BieS}) that
\begin{align}
\label{ExplHerm}
\begin{split}
H_{2n,m,k}(\underline x)&=2^{2n}n!L_n^{\left(k+\frac{m}{2}-1\right)}\left(\vert\underline x\vert^2\right)\\
H_{2n+1,m,k}(\underline x)&=2^{2n+1}n!\underline x L_n^{\left(k+\frac{m}{2}\right)}\left(\vert\underline x\vert^2\right)
\end{split}
\end{align}
and
\begin{align}
\label{ExplGeg}
\begin{split}
C_{2n,m,k}^{(\alpha)}(\underline x)&=2^{2n}n!(\alpha+n+1)_nP_n^{\left(k+\frac{m}{2}-1,\alpha\right)}\left(1-2\vert\underline x\vert^2\right),\\
C_{2n+1,m,k}^{(\alpha)}(\underline x)&=2^{2n+1}n!(\alpha+n+1)_{n+1}\underline x P_n^{\left(k+\frac{m}{2},\alpha\right)}\left(1-2\vert\underline x\vert^2\right),
\end{split}
\end{align}
where $L_n^{(\alpha)}$ and $P_n^{(\alpha,\beta)}$ are the generalized Laguerre and Jacobi polynomials on the real line, respectively; and $(a)_n=a(a+1)\cdots(a+n-1)$ is the Pochhammer symbol. 

One can also define these polynomials using the CK-extension theorem (see e.g. \cite{DSS,S2}). Namely, if we consider the weight functions $\exp\left(-\vert\underline x\vert^2\right)P_k(\underline x)$ and $\left(1-\vert\underline x\vert^2\right)^\alpha P_k(\underline x)$, then it follows from (\ref{CKf}) that
\begin{align}
\mathsf{CK}\left[\exp\left(-\vert\underline x\vert^2\right)P_k(\underline x)\right](x_0,\underline x)&=\exp\left(-\vert\underline x\vert^2\right)\sum_{n=0}^\infty\frac{x_0^n}{n!}H_{n,m}(P_k)(\underline x),\label{CKHerm}\\
\mathsf{CK}\left[\left(1-\vert\underline x\vert^2\right)^\alpha P_k(\underline x)\right](x_0,\underline x)&=\sum_{n=0}^\infty\frac{x_0^n}{n!}\left(1-\vert\underline x\vert^2\right)^{\alpha-n}C_{n,m}^{(\alpha-n)}(P_k)(\underline x).\label{CKGegen}
\end{align}
The main goal of the present paper is to obtain closed formulae for these generating functions with the help of Fueter's theorem. 

The paper is organized as follows. In Section \ref{sect2} we present the results we need on Fueter's theorem. In Section \ref{sect3} we treat the Clifford-Hermite polynomials and in Section \ref{sect4} the Clifford-Gegenbauer polynomials. 

\section{Fueter's theorem}\label{sect2}

For any $\underline x\in\mathbb R^m$ we put $\underline\omega=\underline x/r$, with $r=\vert\underline x\vert$. Another way to prove Fueter's theorem was obtained in \cite{D}. It was based on the calculation of $\mathsf{Ft}\left[h(z),P_k(\underline x)\right]$ in explicit form, i.e.
\begin{multline}\label{goodidea}
\mathsf{Ft}\left[h(z),P_k(\underline x)\right](x_0,\underline x)=(2k+m-1)!!\\
\times\left(D_r\left(k+\frac{m-1}{2}\right)[u(x_0,r)]+\underline\omega\,D^r\left(k+\frac{m-1}{2}\right)[v(x_0,r)]\right)P_k(\underline x),
\end{multline}
where $n!!$ denotes the double factorial of $n$ and where $D_r(n)$, $D^r(n)$ are differential operators defined by
\begin{align*}
D_r(n)[u]&=\left(r^{-1}\partial_r\right)^nu\\
D^r(n)[v]&=\left(\partial_r\,r^{-1}\right)^nv.
\end{align*}
It may be proved by induction that 
\begin{equation}\label{id1}
D_r(n)[u]=\sum_{j=1}^na_{j,n}r^{j-2n}\partial_r^ju,
\end{equation}
where the integers $a_{j,n}$ satisfy: 
\begin{align*}
a_{1,n+1}&=-(2n-1)a_{1,n},\\
a_{j,n+1}&=a_{j-1,n}-(2n-j)a_{j,n},\quad j=2,\dots,n,\\
a_{n+1,n+1}&=a_{n,n}=1,\quad n\ge1.
\end{align*}
It turns out that 
\[a_{j,n}=(-1)^{n+j}\frac{(2n-j-1)!}{2^{n-j}(n-j)!(j-1)!}.\]
In a similar way, we can check that
\begin{equation}\label{id2}
D^r(n)[v]=\sum_{j=0}^nb_{j,n}r^{j-2n}\partial_r^jv,\quad b_{j,n}=a_{j+1,n+1}=(-1)^{n+j}\frac{(2n-j)!}{2^{n-j}(n-j)!j!}.
\end{equation}
It is worth pointing out that the integers $(-1)^{n+j}b_{j,n}$ are the coefficients of the Bessel polynomial of degree $n$ (see \cite{KF}). Other properties of these operators that we shall need are the following
\begin{align}
D_r(n)[fg]&=\sum_{j=0}^n\binom{n}{j}D_r(n-j)[f]D_r(j)[g],\label{lrD1}\\
D^r(n)[fg]&=\sum_{j=0}^n\binom{n}{j}D_r(n-j)[f]D^r(j)[g].\label{lrD2}
\end{align}
It is easily seen that $\mathsf{Ft}\left[h(z),P_k(\underline x)\right]$ defines an $\mathbb R$-linear operator, i.e. 
\[\mathsf{Ft}\left[c_1h_1(z)+c_2h_2(z),P_k(\underline x)\right]=c_1\mathsf{Ft}\left[h_1(z),P_k(\underline x)\right]+c_2\mathsf{Ft}\left[h_2(z),P_k(\underline x)\right],\]
for all $c_1,c_2\in\mathbb R$. Therefore, $\mathsf{Ft}\left[h(z),P_k(\underline x)\right]$ is characterized if we know the action of this operator on the functions $z^n$, $iz^n$, $z^{-n}$ and $iz^{-n}$. This is the subject of the following lemma.

\begin{lem}\label{fident}
For $n\in\mathbb N_0$ the following equalities hold
\[\mathsf{Ft}\left[z^n,P_k(\underline x)\right](x_0,\underline x)\\
=\left\{\begin{array}{ll}0,&n<2k+m-1\\c_{1,n}\mathsf{CK}\left[\underline x^{n-(2k+m-1)}P_k(\underline x)\right](x_0,\underline x),&n\ge2k+m-1,\end{array}\right.\]
\begin{multline*}
\mathsf{Ft}\left[iz^n,P_k(\underline x)\right](x_0,\underline x)\\
=(-1)^{k+\frac{m-1}{2}}(2k+m-1)!!c_{2,n}\mathsf{CK}\left[\frac{\underline x^{n-(2k+m-2)}}{r}P_k(\underline x)\right](x_0,\underline x),
\end{multline*}
and
\begin{align*}
\mathsf{Ft}\left[z^{-n},P_k(\underline x)\right](x_0,\underline x)&=(2k+m-1)!!c_{3,n}\mathsf{CK}\left[\underline x^{-(n+2k+m-1)}P_k(\underline x)\right](x_0,\underline x),\\
\mathsf{Ft}\left[iz^{-n},P_k(\underline x)\right](x_0,\underline x)&=(2k+m-1)!!c_{4,n}\mathsf{CK}\left[\frac{\underline x^{-(n+2k+m-2)}}{r}P_k(\underline x)\right](x_0,\underline x),
\end{align*}
with
\[c_{1,n}=\frac{(-2)^{k+\frac{m-1}{2}}(2k+m-1)!!\left\lfloor\frac{n}{2}\right\rfloor!}{\left(\left\lfloor\frac{n}{2}\right\rfloor-\left(k+\frac{m-1}{2}\right)\right)!},\;\;c_{2,n}=\left\{\begin{array}{ll}\prod_{j=1}^{k+\frac{m-1}{2}}(n-(2j-1)),&n\;\text{even}\\\prod_{j=0}^{k+\frac{m-3}{2}}(n-2j),&n\;\text{odd}\end{array}\right.\]
\[c_{3,n}=\left\{\begin{array}{ll}\displaystyle{\frac{(n+2k+m-3)!!}{(n-2)!!}},&n\;\text{even}\\\displaystyle{\frac{(n+2k+m-2)!!}{(n-1)!!}},&n\;\text{odd}\end{array}\right.,\;\;c_{4,n}=\left\{\begin{array}{ll}\displaystyle{\frac{(n+2k+m-2)!!}{(n-1)!!}},&n\;\text{even}\\\displaystyle{\frac{(n+2k+m-3)!!}{(n-2)!!}},&n\;\text{odd}\end{array}\right.\]
and where $\lfloor\cdot\rfloor$ denotes the floor function.
\end{lem}  
\begin{proof} 
For $h(z)=z^n$ we clearly have that 
\begin{align*}
u(x,y)&=\sum_{j=0}^{\lfloor n/2\rfloor}(-1)^j\binom{n}{2j}x^{n-2j}y^{2j},\\
v(x,y)&=\sum_{j=0}^{\lfloor(n-1)/2\rfloor}(-1)^j\binom{n}{2j+1}x^{n-(2j+1)}y^{2j+1}.
\end{align*}
Using (\ref{goodidea}) it thus follows that
\[\frac{\mathsf{Ft}\left[z^n,P_k(\underline x)\right](x_0,\underline x)\vert_{x_0=0}}{(2k+m-1)!!}=\left\{\begin{array}{ll}\underline\omega^nD_r\left(k+\frac{m-1}{2}\right)[r^n]P_k(\underline x),&n\;\text{even}\\\underline\omega^nD^r\left(k+\frac{m-1}{2}\right)[r^n]P_k(\underline x),&n\;\text{odd.}\end{array}\right.\]
A direct computation then shows that
\begin{multline*}
\frac{\mathsf{Ft}\left[z^n,P_k(\underline x)\right](x_0,\underline x)\vert_{x_0=0}}{(2k+m-1)!!}\\
=\left\{\begin{array}{ll}0,&n<2k+m-1\\\frac{(-2)^{k+\frac{m-1}{2}}\left\lfloor\frac{n}{2}\right\rfloor!}{\left(\left\lfloor\frac{n}{2}\right\rfloor-\left(k+\frac{m-1}{2}\right)\right)!}\,\underline x^{n-(2k+m-1)}P_k(\underline x),&n\ge2k+m-1\end{array}\right.
\end{multline*}
which in view of the CK-extension theorem proves the first equality. The others may be proved in a similar way.
\end{proof}

\section{Generating function for Clifford-Hermite polynomials}\label{sect3}

Consider the holomorphic function
\[h(z)=\exp\left(z^2\right)=\exp\left(x^2-y^2\right)\left(\cos(2xy)+i\sin(2xy)\right).\]
We shall prove that $\mathsf{Ft}\left[\exp\left(z^2\right),P_k(\underline x)\right]$ equals (up to a multiplicative constant) the CK-extension of $\exp\left(-\vert\underline x\vert^2\right)P_k(\underline x)$.

\begin{thm}\label{teoGFHerm}
If $m$ is odd, then a closed formula for the CK-extension of $\exp\left(-\vert\underline x\vert^2\right)P_k(\underline x)$ is given by
\begin{multline*}
\mathsf{CK}\left[\exp\left(-\vert\underline x\vert^2\right)P_k(\underline x)\right](x_0,\underline x)=\frac{\mathsf{Ft}\left[\exp\left(z^2\right),P_k(\underline x)\right](x_0,\underline x)}{(-2)^{k+\frac{m-1}{2}}(2k+m-1)!!}\\
=\exp\left(x_0^2-r^2\right)\left(\sum_{j=0}^{k+\frac{m-1}{2}}\binom{k+\frac{m-1}{2}}{j}(-2)^{-j}D_r(j)[\cos(2x_0r)]\right.\\
\left.+\underline\omega\sum_{j=0}^{k+\frac{m-1}{2}}\binom{k+\frac{m-1}{2}}{j}(-2)^{-j}D^r(j)[\sin(2x_0r)]\right)P_k(\underline x).
\end{multline*}
\end{thm}
\begin{proof} 
From (\ref{goodidea}) we have that
\begin{multline*}
\frac{\mathsf{Ft}\left[\exp\left(z^2\right),P_k(\underline x)\right](x_0,\underline x)}{(2k+m-1)!!}=\left(D_r\left(k+\frac{m-1}{2}\right)\left[\exp\left(x_0^2-r^2\right)\cos(2x_0r)\right]\right.\\
\left.+\underline\omega\,D^r\left(k+\frac{m-1}{2}\right)\left[\exp\left(x_0^2-r^2\right)\sin(2x_0r)\right]\right)P_k(\underline x).
\end{multline*}
Note that 
\[D_r(n)\left[\exp\left(x_0^2-r^2\right)\right]=(-2)^n\exp\left(x_0^2-r^2\right)\]
and using (\ref{lrD1}) and (\ref{lrD2}) we thus obtain
\begin{multline*}
D_r(n)\left[\exp\left(x_0^2-r^2\right)\cos(2x_0r)\right]\\
=\exp\left(x_0^2-r^2\right)\sum_{j=0}^n\binom{n}{j}(-2)^{n-j}D_r(j)[\cos(2x_0r)],
\end{multline*}
\begin{multline*}
D^r(n)\left[\exp\left(x_0^2-r^2\right)\sin(2x_0r)\right]\\
=\exp\left(x_0^2-r^2\right)\sum_{j=0}^n\binom{n}{j}(-2)^{n-j}D^r(j)[\sin(2x_0r)].
\end{multline*}
Therefore
\begin{multline*}
\frac{\mathsf{Ft}\left[\exp\left(z^2\right),P_k(\underline x)\right](x_0,\underline x)}{(2k+m-1)!!}\\
=\exp\left(x_0^2-r^2\right)\left(\sum_{j=0}^{k+\frac{m-1}{2}}\binom{k+\frac{m-1}{2}}{j}(-2)^{k+\frac{m-1}{2}-j}D_r(j)[\cos(2x_0r)]\right.\\
\left.+\underline\omega\sum_{j=0}^{k+\frac{m-1}{2}}\binom{k+\frac{m-1}{2}}{j}(-2)^{k+\frac{m-1}{2}-j}D^r(j)[\sin(2x_0r)]\right)P_k(\underline x),
\end{multline*}
where $D_r(n)\left[\cos(2x_0r)\right]$ and $D^r(n)\left[\sin(2x_0r)\right]$ can be computed using (\ref{id1}) and (\ref{id2}) and are equal to  
\begin{multline*}
D_r(n)\left[\cos(2x_0r)\right]=\sum_{j=1}^n(-1)^{n+j}\frac{(2n-j-1)!}{2^{n-2j}(n-j)!(j-1)!}\frac{x_0^j}{r^{2n-j}}\cos(2x_0r+j\pi/2),
\end{multline*}
\begin{multline*}
D^r(n)\left[\sin(2x_0r)\right]=\sum_{j=0}^n(-1)^{n+j}\frac{(2n-j)!}{2^{n-2j}(n-j)!j!}\frac{x_0^j}{r^{2n-j}}\sin(2x_0r+j\pi/2).
\end{multline*}
In view of the above equalities, we get that
\[\mathsf{Ft}\left[\exp\left(z^2\right),P_k(\underline x)\right](0,\underline x)=(-2)^{k+\frac{m-1}{2}}(2k+m-1)!!\exp\left(-\vert\underline x\vert^2\right)P_k(\underline x).\]
The proof follows using the CK-extension theorem.
\end{proof}

\begin{rem}
For $k=0$, this result was previously established in \cite{DS1}.
\end{rem}

\noindent If, for instance, we take $m=3$ and $P_k(\underline x)=1$ we then can assert that
\begin{multline*}
\mathsf{CK}\left[\exp\left(-\vert\underline x\vert^2\right)\right](x_0,\underline x)=\exp\left(x_0^2-r^2\right)\bigg(\cos(2x_0r)+\frac{x_0}{r}\sin(2x_0r)\\
+\underline\omega\left(\sin(2x_0r)+\frac{\sin(2x_0r)}{2r^2}-\frac{x_0}{r}\cos(2x_0r)\right)\bigg).
\end{multline*}
As a consequence we now obtain the following identities for Laguerre polynomials. They should be compared with the standard generating function
\[(1-t)^{-\alpha-1}\exp\left(-xt/(1-t)\right) = \sum_{n=0}^{\infty} t^n L_n^{(\alpha)}(x),\]
see e.g. \cite{AAR}.

\begin{cor}
For $m$ odd the following identities hold
\begin{align*}
{\rm(i)} \sum_{n=0}^{k+\frac{m-1}{2}}\binom{k+\frac{m-1}{2}}{n}(-2)^{-n}&D_r(n)[\cos(2x_0r)]\\
&=\exp\left(-x_0^2\right)\sum_{n=0}^\infty\frac{2^{2n}n!x_0^{2n}}{(2n)!}L_n^{\left(k+\frac{m}{2}-1\right)}\left(r^2\right),\\
{\rm(ii)}\sum_{n=0}^{k+\frac{m-1}{2}}\binom{k+\frac{m-1}{2}}{n}(-2)^{-n}&D^r(n)[\sin(2x_0r)]\\
&=\exp\left(-x_0^2\right)\sum_{n=0}^\infty\frac{2^{2n+1}n!x_0^{2n+1}}{(2n+1)!}\,rL_n^{\left(k+\frac{m}{2}\right)}\left(r^2\right).
\end{align*}
\end{cor}
\begin{proof} 
The identities easily follow by using (\ref{CKHerm}), (\ref{ExplHerm}) and Theorem \ref{teoGFHerm}.
\end{proof}

\section{Generating function for Clifford-Gegenbauer polynomials}\label{sect4}

In this section we aim at obtaining a closed formula for $\mathsf{CK}\left[\left(1-\vert\underline x\vert^2\right)^\alpha P_k(\underline x)\right]$. But first we need to find the relationship between this generating function and Fueter's theorem, which is given in the next theorem.   

\begin{thm}\label{FvG}
For $m$ odd, it holds that 
\begin{multline*}
\mathsf{Ft}\left[\left(1+z^2\right)^{\alpha+k+\frac{m-1}{2}},P_k(\underline x)\right](x_0,\underline x)=(-2)^{k+\frac{m-1}{2}}(2k+m-1)!!\\
\times\left(\prod_{j=1}^{k+\frac{m-1}{2}}(\alpha+j)\right)\mathsf{CK}\left[\left(1-\vert\underline x\vert^2\right)^\alpha P_k(\underline x)\right](x_0,\underline x).
\end{multline*}
\end{thm}

\begin{proof} 
Using the general Binomial Theorem  
\[(1+z)^\alpha=\sum_{n=0}^\infty\binom{\alpha}{n}z^n,\quad\binom{\alpha}{n}=\frac{1}{n!}\prod_{j=0}^{n-1}(\alpha-j),\quad\vert z\vert<1,\]
we obtain
\begin{multline*}
\mathsf{Ft}\left[\left(1+z^2\right)^{\alpha+k+\frac{m-1}{2}},P_k(\underline x)\right](x_0,\underline x)\\
=\sum_{n=0}^\infty\binom{\alpha+k+\frac{m-1}{2}}{n}\mathsf{Ft}\left[z^{2n},P_k(\underline x)\right](x_0,\underline x).
\end{multline*}
By Lemma \ref{fident} and using again the general Binomial Theorem we can then deduce that
\begin{align*}
&\frac{\mathsf{Ft}\left[\left(1+z^2\right)^{\alpha+k+\frac{m-1}{2}},P_k(\underline x)\right](x_0,\underline x)}{(2k+m-1)!!}\\
&=(-2)^{k+\frac{m-1}{2}}\sum_{n=0}^\infty\binom{\alpha+k+\frac{m-1}{2}}{n+k+\frac{m-1}{2}}\frac{\left(n+k+\frac{m-1}{2}\right)!}{n!}\mathsf{CK}\left[\underline x^{2n}P_k(\underline x)\right](x_0,\underline x)\\
&=(-2)^{k+\frac{m-1}{2}}\left(\prod_{j=1}^{k+\frac{m-1}{2}}(\alpha+j)\right)\mathsf{CK}\left[\sum_{n=0}^\infty\binom{\alpha}{n}\underline x^{2n}P_k(\underline x)\right](x_0,\underline x)\\
&=(-2)^{k+\frac{m-1}{2}}\left(\prod_{j=1}^{k+\frac{m-1}{2}}(\alpha+j)\right)\mathsf{CK}\left[\left(1-\vert\underline x\vert^2\right)^\alpha P_k(\underline x)\right](x_0,\underline x),
\end{align*}
which proves the desired result.
\end{proof}

Thus we have reduced the problem of finding the generating function to computing 
\[\mathsf{Ft}\left[\left(1+z^2\right)^{\alpha+k+\frac{m-1}{2}},P_k(\underline x)\right].\]
It is easy to see that if $h(z)=u(x,y)+iv(x,y)$ is a holomorphic function, then
\begin{align*}
\partial_y^nu&=\frac{i^n}{2}(\partial_z^nh+(-1)^n\partial_{\overline z}^n\overline h)\\
\partial_y^nv&=-\frac{i^{n+1}}{2}(\partial_z^nh+(-1)^{n+1}\partial_{\overline z}^n\overline h)
\end{align*}
where $\partial_z=\frac{1}{2}(\partial_x-i\partial_y)$ is the complex derivative. Observe also that
\[\partial_z^n\left(1+z^2\right)^\beta=\left(1+z^2\right)^{\beta-n}Q_n^{(\beta)}(z),\quad\beta\in\mathbb R,\]
where $Q_n^{(\beta)}(z)$ is a polynomial in $z$ with real coefficients which satisfy
\[Q_{n+1}^{(\beta)}(z)=2(\beta-n)zQ_n^{(\beta)}(z)+\left(1+z^2\right)\partial_zQ_n^{(\beta)}(z),\quad Q_0^{(\beta)}(z)=1.\]

\begin{rem}
Note that $Q_n^{(\beta)}(z)$ equals (up to a multiplicative constant) 
\[i^nC_n^{\left(\beta-n+\frac{1}{2}\right)}(iz),\]
where $C_n^{(\beta)}$ are the Gegenbauer polynomials on the real line.
\end{rem}

Using the previous equalities together with Theorem \ref{FvG} as well as identities (\ref{goodidea}), (\ref{id1}) and (\ref{id2}), we arrive at the following result:

\begin{thm}\label{teoGFGegen}
Suppose that $\alpha\in\mathbb R\setminus\left\{-1,-2,\dots,-k-\frac{m-1}{2}\right\}$. If $m$ is odd, then a closed formula for the CK-extension of $\left(1-\vert\underline x\vert^2\right)^\alpha P_k(\underline x)$ is given by 
\begin{multline*}
2^{k+\frac{m+1}{2}}\left(\prod_{n=1}^{k+\frac{m-1}{2}}(\alpha+n)\right)\mathsf{CK}\left[\left(1-\vert\underline x\vert^2\right)^\alpha P_k(\underline x)\right](x_0,\underline x)\\
=\left(\sum_{n=1}^{k+\frac{m-1}{2}}\frac{a_{n,m,k}}{r^{2k+m-n-1}}\left(\left(1+Z^2\right)^{\beta-n}Q_n^{(\beta)}\left(Z\right)+(-1)^{n}\left(1+\overline Z^2\right)^{\beta-n}Q_n^{(\beta)}\left(\overline Z\right)\right)+\right.\\
\left.\underline\omega\sum_{n=0}^{k+\frac{m-1}{2}}\frac{b_{n,m,k}}{r^{2k+m-n-1}}\left(\left(1+Z^2\right)^{\beta-n}Q_n^{(\beta)}(Z)+(-1)^{n+1}\left(1+\overline Z^2\right)^{\beta-n}Q_n^{(\beta)}\left(\overline Z\right)\right)\right)P_k(\underline x),
\end{multline*}
where $\beta=\alpha+k+\frac{m-1}{2}$, $Z=x_0+ir$ and 
\[a_{n,m,k}=\frac{(-i)^{n}(2k+m-n-2)!}{(2k+m-2n-1)!!\,(n-1)!},\;\;b_{n,m,k}=\frac{(-i)^{n+1}(2k+m-n-1)!}{(2k+m-2n-1)!!\,n!}.\]
\end{thm}
If, for instance, we take $m=3$ and $P_k(\underline x)=1$ we then can assert that
\begin{multline*}
\mathsf{CK}\left[\left(1-\vert\underline x\vert^2\right)^\alpha\right](x_0,\underline x)=-\frac{i}{2r}\left(Z\left(1+Z^2\right)^\alpha-\overline Z\left(1+\overline Z^2\right)^\alpha\right)\\
-\frac{\underline\omega}{2r}\left(Z\left(1+Z^2\right)^\alpha+\overline Z\left(1+\overline Z^2\right)^\alpha+\frac{i}{2(\alpha+1)r}\left(\left(1+Z^2\right)^{\alpha+1}-\left(1+\overline Z^2\right)^{\alpha+1}\right)\right).
\end{multline*}

Again we find as a consequence the following identities for Jacobi polynomials, which should be compared with the standard generating function from e.g. \cite{AAR} given by
\[2^{\alpha + \beta} R^{-1} (1-t+R)^{-\alpha}(1+t+R)^{-\beta}= \sum_{n=0}^{\infty}t^n P_n^{(\alpha,\beta)} (x), \quad R = (1-2 x t + t^2)^{1/2}.\]

\begin{cor}
For $m$ odd the following identities hold
\begin{align*}
{\rm(i)} &\sum_{n=1}^{k+\frac{m-1}{2}}\frac{a_{n,m,k}}{r^{2k+m-n-1}}\left(\left(1+Z^2\right)^{\beta-n}Q_n^{(\beta)}\left(Z\right)+(-1)^{n}\left(1+\overline Z^2\right)^{\beta-n}Q_n^{(\beta)}\left(\overline Z\right)\right)\\
&=M\sum_{n=0}^\infty\frac{2^{2n}n!(\alpha-n+1)_nx_0^{2n}}{(2n)!}\left(1-r^2\right)^{\alpha-2n}P_n^{\left(k+\frac{m}{2}-1,\alpha-2n\right)}\left(1-2r^2\right),
\end{align*}
\begin{align*}
{\rm(ii)} &\sum_{n=0}^{k+\frac{m-1}{2}}\frac{b_{n,m,k}}{r^{2k+m-n-1}}\left(\left(1+Z^2\right)^{\beta-n}Q_n^{(\beta)}(Z)+(-1)^{n+1}\left(1+\overline Z^2\right)^{\beta-n}Q_n^{(\beta)}\left(\overline Z\right)\right)\\
&=M\sum_{n=0}^\infty\frac{2^{2n+1}n!(\alpha-n)_{n+1}x_0^{2n+1}}{(2n+1)!}\left(1-r^2\right)^{\alpha-2n-1}rP_n^{\left(k+\frac{m}{2},\alpha-2n-1\right)}\left(1-2r^2\right),
\end{align*}
where $M=2^{k+\frac{m+1}{2}}\prod_{n=1}^{k+\frac{m-1}{2}}(\alpha+n)$.
\end{cor}
\begin{proof}
The identities easily follow by using (\ref{CKGegen}), (\ref{ExplGeg}) and Theorem \ref{teoGFGegen}.
\end{proof}

\subsection*{Acknowledgments}

D. Pe\~na Pe\~na acknowledges the support of a Postdoctoral Fellowship funded by the \lq\lq Special Research Fund" (BOF) of Ghent University.

\end{document}